\documentclass[11pt]{amsart}

\usepackage{hyperref}
\usepackage{mathtools}
\usepackage{amsthm}
\usepackage{amssymb}
\usepackage{amsfonts}
\usepackage[utf8]{inputenc}
\usepackage{float}
\usepackage{verbatim}
\usepackage{enumitem}
\usepackage{booktabs}

\usepackage{tikz}

\newtheorem{theorem}{Theorem}[section]

\newtheorem{lemma}[theorem]{Lemma}

\theoremstyle{definition}

\theoremstyle{remark}

\numberwithin{equation}{section}

\newcommand{\bd}{\partial}

\DeclareMathOperator{\area}{\mathrm{Area}}
\DeclareMathOperator{\length}{\mathrm{Length}}

\newcommand{\Sph}{\mathbb{S}}

\newcommand{\Z}{\mathbb{Z}}

\newcommand{\del}{\partial}

\newcommand{\D}{\mathcal{D}}

\newcommand{\T}{\mathbb{T}}

\newcommand{\R}{\mathbb{R}}

\definecolor{light-gray}{gray}{.95}
\definecolor{dark-gray}{gray}{.7}

\begin{document}
\title[Second variation of eigenvalues]{Second variation formula for eigenvalue functionals on surfaces}
\author[M.~Karpukhin]{Mikhail~Karpukhin}
\date{}
\address{Department of Mathematics, University College London, 25 Gordon Street, London, WC1H 0AY, UK} \email{m.karpukhin@ucl.ac.uk}



\begin{abstract}
Consider the first nontrivial eigenvalue of the Laplacian on a closed surface as a functional on the space of Riemannian metrics of unit area.
N.~Nadirashvili has discovered a remarkable connection between critical points of this functional and minimal surfaces in the sphere. It was later extended by A. El Soufi and S. Ilias to cover $k$-th eigenvalues and critical points in a fixed conformal class, where the latter correspond to harmonic maps to the sphere. These results, however, only contain first order information and cannot be used to determine whether a given critical metric a local maximiser or not. In the present paper we write down the second variation formula for critical metrics and show that the flat metric on the non-rhombic torus can never be a conformal maximiser for the first eigenvalue. Analogous results are proved in the context of the Steklov eigenvalues and flat metrics on a cylinder.
\end{abstract}

\maketitle

\section{Introduction}

\subsection{Laplace eigenvalues}
Given a closed Riemannian surface $(M,g)$ we consider the eigenvalues of the Laplace-Beltrami operator
\[
0=\lambda_0(M,g)<\lambda_1(M,g)\leq\lambda_2(M,g)\leq\lambda_3(M,g)\leq\ldots\nearrow+\infty
\]
as functionals of the metric $g$. The problem of geometric optimisation of eigenvalues consists in determining the following quantities
\[
\Lambda_k(M,[g]) = \sup_{h\in [g]}\bar\lambda_k(M,h):= \sup_{h\in [g]}\lambda_k(M,h)\area(M,h);
\] 
\[
\Lambda_k(M) = \sup_g\bar\lambda_k(M):=\sup_{g}\lambda_k(M,g)\area(M,g),
\]
where $[g] = \{e^{2\omega}g|\,\,\omega\in C^\infty(M)\}$ is the conformal class of the metric $g$. These quantities are known to be finite, see~\cite{Korevaar, Has}, but finding their exact values is a challenging problem in general. The first results in this area were the computations of $\Lambda_1(\mathbb{S}^2)$ by Hersch~\cite{Hersch} and $\Lambda_1(\mathbb{RP}^2)$ by Li-Yau~\cite{LiYau}. In both cases the maximal metric is the metric of constant curvature and the arguments are substantially simplified by the fact that on $\Sph^2$ and $\mathbb{RP}^2$ there is a unique conformal class up to a diffeomorphism, which is not the case for other surfaces. A big step forward was made by Nadirashvili in his seminal work~\cite{NadirashviliT2}, where he first established the connection between the geometric eigenvalue optimisation problem and the theory of minimal surfaces. He then used it to prove that the metric on the flat equilateral torus is the unique maximiser for $\Lambda_1(\T^2)$, see also the discussion in~\cite{CKM}. The connection to minimal surfaces has been the driving force behind the recent advances in the field, both as a motivation and as a tool for computing $\Lambda_k(M)$. For an incomplete list of achievements in this area we refer to the introduction to~\cite{KKMS}.

To be more precise, in~\cite{NadirashviliT2} Nadirashvili proved a correspondence between $\bar\lambda_1$-critical metrics and minimal immersions to the sphere. In particular, he showed that the $\lambda_1$-eigenspace of a $\bar\lambda_1$-critical metric contains a collection of eigenfunctions forming a minimal homothetic immersion of the surface into the unit sphere $\Sph^n\subset\R^{n+1}$. This result was later generalised in~\cite{ESIextremal} by El Soufi and Ilias to cover higher eigenvalues and the so-called $\bar\lambda_k$-conformally critical metrics, i.e. metrics that are critical for $\bar\lambda_k$ within a fixed conformal class of metrics. It is shown that $\lambda_k$-eigenspace of a $\bar\lambda_k$-critical metric contains a collection of eigenfunctions forming a harmonic map from the surface to the unit sphere $\Sph^n\subset\R^{n+1}$, which amounts to the existence of $\lambda_k$-eigenfunctions $(u_1,\dots, u_{n+1})$ satisfying $\sum_{i=1}^{n+1}u_i^2\equiv 1$. These criticality conditions are essentially Euler-Lagrange equations for the Lipschitz functional $g\mapsto \bar\lambda_k(M,g)$ and, as a result, they do not yield any information on whether the metric in question is a maximiser or not. In the present note we write down the folklore expression for the second variation of the functional $\bar\lambda_k(M,g)$ at a critical point, see Theorem~\ref{thm:2nd_var}, and use it to show that certain known flat $\bar\lambda_1$-conformally critical metrics on $\T^2$ cannot be $\bar\lambda_1$-conformally maximal. The main downside of the second variation formula~\eqref{eq:2nd_var} is that it requires a good understanding of the full spectrum rather than just $\lambda_k$-eigenspace, which is why flat metrics prove to be particularly convenient for our purposes.

\subsection{Conformal classes on the torus}  
Despite the fact that the value of $\Lambda_1(\T^2)$ is known by the work of Nadirashvili~\cite{NadirashviliT2}, the conformal $\bar\lambda_1$-maxima on $\T^2$ are not completely understood.
It is known that the space of conformal classes on $\T^2$ is parametrised by the flat metrics $g_{a,b}$ induced by the factorisation of $\R^2$ by the lattice $\Gamma_{a,b} = \Z(1,0) + \Z(a,b)$, $a^2+b^2\geq 1$, $0\leq a\leq 1/2$. As $\lambda_1(g_{a,b})$-eigenspace always contains $\sin\left(\frac{2\pi}{b}y\right)$ and $\cos\left(\frac{2\pi}{b}y\right)$, the result of~\cite{ESIextremal} implies that all those metrics are $\bar\lambda_1$-conformally critical. Furthermore, it was proved by El Soufi, Ilias and Ros~\cite{ESIR} that if $a^2+b^2 = 1$, then $\Lambda_1(\T^2,[g_{a,b}]) = \bar\lambda_1(g_{a,b})$. At the same time, $\bar\lambda_1(g_{a,b}) = \frac{4\pi^2}{b}$, hence, for $b\geq \frac{\pi}{2}$ one has 
\[
\bar\lambda_1(g_{a,b})\leq 8\pi<\Lambda_1(\T^2,[g_{a,b}]),
\]
where the latter inequality follows from a general result of Petrides~\cite{Petrides} who proved that for $M\ne\Sph^2$ one always has $\Lambda_1(M,[g])>8\pi$. In the remaining cases $a^2+b^2 > 1$, $b<\frac{\pi}{2}$ it was not known until now whether $g_{a,b}$ is $\bar\lambda_1$-conformally maximal.

\begin{theorem} 
\label{thm:torus}
For $a^2+b^2>1$, $0\leq a\leq 1/2$ the metric $g_{a,b}$ is not $\bar\lambda_1$-conformally maximal.
\end{theorem}

This theorem is proved by choosing an appropriate perturbation for which the second variation is positive. The condition $a^2+b^2>1$ ensures that $\lambda_1(g_{a,b})$-eigenspace is $2$-dimensional, which has a significant effect on the second variation formula~\eqref{thm:2nd_var}.
Finally, we remark that the numerical computations in~\cite{KOO} suggest that for all $a^2+b^2 > 1$ the $\bar\lambda_1$-conformally maximal metric is rotationally symmetric in the direction of the vector $(a,b)$.

\subsection{Steklov eigenvalues}
Let $(N,g)$ be a surface with boundary. The number $\sigma$ is called a Steklov eigenvalue if there exists a nontrivial solution of the following equation
\begin{equation*}
\begin{cases}
\Delta u = 0\quad&\text{ on $N$};\\
\del_n u = \sigma u\quad&\text{ on $\del N$}.
\end{cases}
\end{equation*}
The Steklov eigenvalues form a sequence
\[
0=\sigma_0(N,g)<\sigma_1(N,g)\leq\sigma_2(N,g)\leq\sigma_3(N,g)\leq\ldots\nearrow+\infty.
\]
For a recent review of various aspects of spectral geometry of Steklov eigenvalues we refer to~\cite{CGGS}.
In the present paper we consider the following normalised quantities
\[
\Sigma_k(N,[g]) = \sup_{h\in [g]}\bar\sigma_k(N,h):= \sup_{h\in [g]}\sigma_k(N,h)\length(\del N,h);
\] 
\[
\Sigma_k(N) = \sup_g\bar\sigma_k(N):=\sup_{g}\sigma_k(N,g)\length(\del N,g).
\]
The geometric optimisation problem refers to determining the exact value of these quantities. In this context Steklov eigenvalues are often seen as counterparts of Laplace eigenvalues for surfaces with boundary. One reason is that there are a lot of empirical similarities between the two problems as one can often prove analogous results even if the proofs do not always carry over. The more important reason is that critical metrics for Steklov eigenvalues also correspond to minimal surfaces albeit in a different context. In line with these analogies, in the present paper we study the second variation for Steklov-critical metrics and study the maximality of flat metrics on cylinders.
 
 More specifically, it is proved by Fraser and Schoen in~\cite{FS:extremal} that a metric $g$ on $N$ is $\bar\sigma_k$-conformally critical if there exists a collection of $\sigma_k$-eigenfunctions $(u_1,\ldots,u_{n})$ satisfying $\sum_{i=1}^n u_i^2\equiv 1$ on $\del N$. Once again, this condition does not tell us anything about whether $g$ is a maximiser or not. Our first result is Theorem~\ref{thm:2nd_var_St}, where we obtain formula~\eqref{eq:2nd_var_St} for the second variation of  $\bar\sigma_k$. This formula once again involves the complete Steklov spectrum of $g$ and, thus, is difficult to use in general. However, as for the flat metrics on a torus, in the Steklov case flat metrics on a cylinder represent a convenient situation, where formula~\eqref{eq:2nd_var_St} can be effectively applied.

\subsection{Conformal classes on the cylinder} The space of  conformal classes on the annulus $\mathbb{A}$ can be parametrised by the flat metrics $g_T$ on $\mathbb{S}^1\times [-T,T]$, where $\mathbb{S}^1$ is a circle of length $2\pi$. Fraser and Schoen proved in~\cite{FS} that $\Sigma_1(\mathbb{A})$ is achieved on $g_{T_1}$, where $T_1\approx 1.2$ is the unique positive solution of $\coth(T) = T$. In particular, this implies that $\Sigma_1(\mathbb{A}) = \Sigma_1(\mathbb{A},[g_{T_1}]) = \bar\sigma_1(\mathbb{A}, g_{T_1})$, but the value of $\Sigma_1(\mathbb{A},[g_{T}])$ is unknown for $T\ne T_1$. In~\cite{KM} it is conjectured that $T>T_1$ this value is achieved on an explicitly described rotationally symmetric metric, whereas for $T<T_1$ the $\bar\sigma_1$-conformal maximisers are not rotationally symmetric. In this paper we confirm the latter statement. For $T<T_1$ the $\sigma_1(\mathbb{A},g_T)$-eigenspace is spanned by $\sin\theta\sinh t$ and $\cos\theta\sinh t$, which clearly implies that $g_T$ is $\bar\sigma$-conformally critical. 
Moreover, in~\cite{KM} it is proved that $g_T$ are in fact the only rotationally symmetric $\bar\sigma_1$-conformally critical metrics for $T<T_1$. 

\begin{theorem}
\label{thm:cylinder}
 Let $T_1$ be the unique solution of $\coth T = T$. Then for $T<T_1$ the flat metric $g_T$ on $\mathbb{S}^1\times [-T,T]$ is not $\bar\sigma_1$-conformally maximal. In particular, for such $T$ there are no rotationally symmetric $\bar\sigma_1$-conformally maximal metrics.
\end{theorem}

The theorem is proved by choosing a particular perturbation for which the second variation is positive. The condition $T<T_1$ ensures that the $\sigma_1(g_T)$-eigenspace is $2$-dimensional.  

\subsection{Discussion} Theorems~\ref{thm:2nd_var} and~\ref{thm:2nd_var_St} appear to be folklore results, but we were not able to find an appropriate reference for them in the context of general background metrics. It would be natural to assume that the second variation formulae (or at least its rough form and the way to obtain them) are known to members of the community, who were in turn deterred from pursuing them further by the dependence of the formulae on the full spectrum of the metric. At the very least this can certainly be said about the author of the present article. Our main contribution here is showing that the second variation can be useful. We hope that the present paper will inspire other researchers to take a closer look and find other applications, possibly to proving local maximality properties of the critical metrics. 

We conclude with the remark that there is an alternative approach to stability properties of critical metrics based on the connection to harmonic maps and minimal surfaces, see~\cite{KNPS}. It connects the second variations of energy and area respectively to the local stability properties of critical metrics. In the form presented in~\cite{KNPS} it cannot be used to prove non-maximality of a critical point, but it is well-suited for proving local maximality. It would be interesting to see if the combination of the two approaches could lead to the improved understanding of the local behaviour of eigenvalue functional near critical points. 

\subsection*{Acknowledgements} A huge thanks goes to Nikolai Nadirashvili for his hospitality, support and guidance during the author's early career stages.The author is grateful to the referees for the valuable comments on the initial version of the manuscript.

\section{Second variation for the Laplace eigenvalues}
\begin{theorem}
\label{thm:2nd_var}
Let $E_k$ denote the $\lambda_k$-eigenspace. Suppose that $\omega\in C^\infty(M)$ is such that 
\begin{equation}
\label{eq:2nd_var_cond}
\int_M\omega uv\,dv_g = 0
\end{equation}
for any $u,v\in E_k$. If $\lambda_{k-1}<\lambda_k$, then 
\begin{equation}
\label{eq:2nd_var}
\lambda_k(e^{t\omega}g) = \lambda_k(g) + \alpha t^2 + o(t^2)
\end{equation}
as $t\to 0$, where $\alpha =\lambda_k( \|\omega\|^2_{L^2} + \mu\area(M,g))$ and $\mu$ is the smallest eigenvalue of the following quadratic form on $E_k$
\[
Q_{\omega}(u,u) = - \sum_{\lambda_i\ne\lambda_k}\frac{\lambda_i+\lambda_k}{\lambda_i-\lambda_k}\|\Pi_i(\omega u)\|^2_{L^2},
\]  
where $\Pi_i$ is the $L^2$-orthogonal projection onto $\lambda_i$-eigenspace and the sum on the r.h.s. is over distinct eigenvalues of $\Delta_g$.
\end{theorem}
\begin{proof}
Suppose that  $\lambda_{k-1}(g)<\lambda_k(g)=\ldots=\lambda_{k+m}(g)<\lambda_{k+m+1}(g)$.  
It is shown in~\cite{ESIextremal} that we can apply the perturbation theory of Kato~\cite{Kato} to the situation at hand. In particular, perturbing $g = g_0$ as $g_t = e^{\omega t}g$ the $m+1$-dimensional $\lambda_k(g)$-eigenspace splits into a collection of analytic branches of eigenvalues and eigenfunctions, which we denote by $(\lambda_j(t), u_j(t))$. Note that in general it is not possible to enumerate these branches so that  $\lambda_j(t)$ is the $j$-th eigenvalue of $g_t$ for all values of $t$, which is the case for example if the first order terms in the asymptotic expansions of $\lambda_j(t)$ do not vanish. At the same time, according to the computations of~\cite{ESIextremal}, the condition~\eqref{eq:2nd_var_cond} implies that the first order term in the expansion of $\lambda_j(t)$ does indeed vanish.

To compute the second order term in $\lambda_j(t)$ we begin by computing the first order term in the expansion of $u_j(t)$. To do that we differentiate $\Delta_{g_t}u_j(t) = \lambda_j(t) u_j(t)$ at $t=0$ and use the conformal covariance property $\Delta_{g_t} = e^{-\omega t}\Delta_g$. Thus, using that $\dot \lambda_j =0$, we obtain
\[
-\lambda_k\omega u_j + \Delta_g \dot u_j = \lambda_k \dot u_j.
\]
Taking the projection of this equality onto $\lambda_i$-eigenspace $\lambda_i\ne\lambda_k$ one obtains
\[
-\lambda_k\Pi_i(\omega u_j) + \lambda_i\Pi_i(\dot u_j) = \lambda_k\Pi_i(\dot u_j)
\]
or, equivalently,
\[
\dot u_j = \sum_{\lambda_i\ne\lambda_k}\frac{\lambda_k}{\lambda_i-\lambda_k}\Pi_i(\omega u_j),
\]
where the summation is over {\em distinct} eigenvalues of $\Delta_g$. Note that $\Pi_k(\omega u_j)= 0$ by~\eqref{eq:2nd_var_cond}.

Taking two derivatives of $\Delta_{g_t}u_j(t) = \lambda_j(t) u_j(t)$ at $t=0$  and using $\dot \lambda_j =0$, we obtain
\[
\lambda_k\omega^2 u_j - 2\omega \sum_{\lambda_i\ne\lambda_k}\frac{\lambda_i\lambda_k}{\lambda_i - \lambda_k}\Pi_i(\omega u_j) + \Delta_g \ddot u_j = \ddot\lambda_j u_j + \lambda_k\ddot u_j.
\]
Taking the projection onto $\lambda_k$-eigenspace we find that $u_j$ is $\ddot\lambda_j$-eigenvector of an operator
\[
u_j\mapsto \lambda_k\left(\Pi_k(\omega^2 u_j) - 2\sum_{\lambda_i\ne\lambda_k}\frac{\lambda_i}{\lambda_i-\lambda_k}\Pi_k(\omega\Pi_i(\omega u_j))\right),
\]
defined on the $\lambda_k$-eigenspace. The associated quadratic form is 
\[
\lambda_kQ_\omega(u,u) = \lambda_k\left(\|\omega u\|^2_{L^2} - 2\sum_{\lambda_i\ne\lambda_k}\frac{\lambda_i}{\lambda_i-\lambda_k} \|\Pi_i(\omega u)\|^2_{L^2} \right)
\]
and using Parseval's identity and the fact that~\eqref{eq:2nd_var_cond} implies $\Pi_k(\omega u) = 0$ for any $\lambda_k$-eigenfunction $u$, we arrive at
\[
Q_\omega(u,u) = - \sum_{\lambda_i\ne\lambda_k}\frac{\lambda_i+\lambda_k}{\lambda_i-\lambda_k}\|\Pi_i(\omega u)\|^2_{L^2}.
\]
In particular, if $\lambda_{k-1}<\lambda_k$, then $\ddot\lambda_k = \lambda_k\mu$, where $\mu$ is the smallest eigenvalue of $Q_\omega$.

 It remains to compute the second derivative of $\bar\lambda_k(t) = \lambda_k(t)\area(M,g_t)$. Since $\dot\lambda_k=0$, this derivative equals
 \[
\ddot\lambda_k\area(M,g) + \lambda_k\ddot{\area}(M,g) = \lambda_k(\mu\area(M,g) + \|\omega\|_{L^2}^2)
 \]
as claimed.
\end{proof}

Let us now turn to the proof of Theorem~\ref{thm:torus}.

%
\begin{proof}[Proof of Theorem~\ref{thm:torus}]
To prove the theorem we will find $\omega$ satisfying the assumptions of Theorem~\ref{thm:2nd_var} for which $\alpha$ in equation~\eqref{eq:2nd_var} is positive. 

It is more convenient to work with the unit area metric $g = \frac{1}{\sqrt{b}}g_{a,b}$. Set $\omega = \sqrt{2}\sin\left(\frac{2\pi}{b}y\right)$, so that $\|\omega\|_{L^2(g)} = 1$. Then the functions $u_1 = \sqrt{2}\sin\left(\frac{2\pi}{b}y\right)$ and $u_2 = \sqrt{2}\cos\left(\frac{2\pi}{b}y\right)$ form an orthonormal basis of $\lambda_1$-eigenspace (the assumption $a^2+b^2>1$ is used here) and one has 
\begin{align*}
&\omega u_1 =  2\sin^2\left(\frac{2\pi}{b}y\right) = 1 - \frac{1}{\sqrt 2}\left(\sqrt{2}\cos\left(\frac{4\pi}{b}y\right)\right); \\
&\omega u_2 = 2\sin\left(\frac{2\pi}{b}y\right)\cos\left(\frac{2\pi}{b}y\right) = \frac{1}{\sqrt 2}\left(\sqrt{2}\sin\left(\frac{4\pi}{b}y\right)\right),
\end{align*}
where $\sqrt{2}\sin\left(\frac{4\pi}{b}y\right)$ and $\sqrt{2}\cos\left(\frac{4\pi}{b}y\right)$ are orthonormal functions in the $4\lambda_1$-eigenspace. Hence, one has that $\omega u_i$ is orthogonal to $E_1$ for $i=1,2$ and
\begin{align*}
&Q_\omega(u_1,u_1) = 1 - \frac{4\lambda_1+\lambda_1}{4\lambda_1-\lambda_1}\cdot\frac{1}{2} = \frac{1}{6};\\
&Q_\omega(u_1,u_2) = 0;\\
&Q_\omega(u_2,u_2) = - \frac{4\lambda_1+\lambda_1}{4\lambda_1-\lambda_1}\cdot\frac{1}{2} = -\frac{5}{6}.
\end{align*}
Hence $\mu = -5/6$ and $\alpha = \lambda_1/6>0$.\qedhere

\end{proof}

\section{Second variation for the Steklov eigenvalues}

We recall that the Steklov eigenvalues coincide with the eigenvalues of the Dirichlet-to-Neumann map $\D_g\colon C^\infty(N)\to C^\infty(N)$ given by
\[
\D_g(u) = \partial_n(\hat u),
\]
where $\hat u$ is the unique harmonic extension of $u$ to the interior on $N$. It is known that $\D_g$ is a non-negative elliptic pseudodifferential operator of order $1$, see~\cite{CGGS}. It is also easy to see that the operator is conformally covariant, $\D_{e^{2\omega}g} = e^{-\omega}\D_g$.

\begin{theorem}
\label{thm:2nd_var_St}
Let $E_k$ denote the $\sigma_k$-eigenspace. Suppose that $\omega\in C^\infty(M)$ is such that 
\[
\int_{\bd N}\omega uv\,dv_g = 0
\]
for any $u,v\in E_k$. If $\sigma_{k-1}<\sigma_k$, then 
\begin{equation}
\label{eq:2nd_var_St}
\sigma_k(e^{2t\omega}g) = \sigma_k(g) + \alpha t^2 + o(t^2)
\end{equation}
as $t\to 0$, where $\alpha =\sigma_k( \|\omega\|^2_{L^2(\bd N)} + \mu\length(\bd N,g))$ and $\mu$ is the smallest eigenvalue of the following quadratic form on $E_k$
\[
Q_{\omega}(u,u) = - \sum_{\sigma_i\ne \sigma_k}\frac{\sigma_i+\sigma_k}{\sigma_i-\sigma_k}\|\Pi_i(\omega u)\|^2_{L^2},
\]  
where $\Pi_i$ is the $L^2$-orthogonal projection onto $\sigma_i$-eigenspace and the sum on the r.h.s. is over distinct eigenvalues of the Dirichlet-to-Neumann map.
\end{theorem}
\begin{proof}
The proof is analogous to the proof of Theorem~\ref{thm:2nd_var}. Indeed, the important points of the proof were Kato's perturbation theory, which can still be applied to $\D_g$, and conformal covariance of $\Delta_g$. The details are left to the reader.
\end{proof}

\begin{proof}[Proof of Theorem~\ref{thm:cylinder}]

%
As before we will present $\omega$ satisfying the assumptions of Theorem~\ref{thm:2nd_var_St} such that $\alpha$ in equation~\eqref{eq:2nd_var_St} is positive.

For $T<T_1$ the $\sigma_1$-eigenspace $E_1$ of $\D_{g_T}$ is given by $u_1 = \sin\theta$ for $t=\pm T$ and $u_2 = \cos\theta$ for $t=\pm T$, $\|u_1\|^2 = \|u_2\|^2 = 2\pi$, see e.g.~\cite{FSconformal, KM}. Take $\omega = \sin\theta - a\sin3\theta$ for $t=\pm T$, where $a>0$ is to be chosen later. Observe that $E_1$ and $\omega$ lie in the subspace of functions invariant with respect to $T\leftrightarrow (-T)$, the same is true for their various products and linear combinations, hence, in the following we omit the phrase ``for $t = \pm T$'' and simply write e.g. $u_1 = \sin\theta$ meaning that values of all functions are the same on both boundary components. One has that $\|\omega\|_{L^2}^2 = 2\pi(1+a^2)$, $\length(\bd N,g) = 4\pi$ and 
\begin{align*}
&\omega u_1 = \sin\theta(\sin\theta-a\sin3\theta) = \frac{1}{2}(1 - (1+a)\cos2\theta + a\cos4\theta)\\
&\omega u_2 = \cos\theta(\sin\theta-a\sin3\theta) = \frac{1}{2}((1-a)\sin2\theta - a\sin4\theta).
\end{align*}
Therefore, the conditions of Theorem~\ref{thm:2nd_var_St} are satisfied, and $\omega u_1\perp \omega u_2$. Defining 
\[
b_2(T) = \frac{2\tanh(2T) + \tanh(T)}{2\tanh(2T) - \tanh(T)}\quad\text{and}\quad b_4(T) = \frac{4\tanh(4T) + \tanh(T)}{4\tanh(4T) - \tanh(T)}
\]
one, furthermore, obtains
\begin{align*}
& Q_\omega(u_1,u_1) = \frac{\pi}{2}\left(2 -(1+a)^2b_2(T) - a^2 b_4(T) \right) =: 2\pi \mu_1(a);\\
& Q_\omega(u_1,u_2) = 0;\\
& Q_\omega(u_2,u_2) = -\frac{2\pi}{4}((1-a)^2b_2(T) + a^2b_4(T)) =:2\pi \mu_2(a),
\end{align*}
where $\mu (a)= \min\{\mu_1(a),\mu_2(a)\}$. As a result,
\[
\alpha(a) = 2\pi\sigma_1((1+a^2) + 2\mu(a)).
\]
Thus, the theorem follows from the following lemma.
\begin{lemma}
\label{lem:a0}
There exists $a_0\approx 0.2$ such that $\alpha(a_0)>0$.
\end{lemma}
\begin{proof}[Proof of Lemma~\ref{lem:a0}] First, let us record that 
\begin{align*}
\mu_1(a) =  \frac{1}{2} - \frac{(1+a)^2}{4}b_2(T) - \frac{a^2}{4}b_4(T);\\
\mu_2(a) = -\frac{1}{4}\left((1-a)^2b_2(T) + a^2b_4(T)\right).
\end{align*}
Let us observe that $b_2(T)$ and $b_4(T)$ are increasing functions. Indeed,
\[
b_2(T) = 1 + \frac{2}{2\tanh(2T)\coth(T)-1},\quad b_4(T) = 1 + \frac{2}{4\tanh(4T)\coth(T)-1}
\]
then
\[
\frac{d}{dT}\tanh(2T)\coth(T) = \frac{\sinh(2T) - \sinh(2T)\cosh(2T)}{\cosh^2(2T)\sinh^2(T)}<0;
\]
\[
\frac{d}{dT}\tanh(2T)\coth(T) = \frac{2\sinh(2T) - \sinh(4T)\cosh(4T)}{\cosh^2(4T)\sinh^2(T)}<0,
\]
where in the last inequality we used that 
\[
2\sinh(2T)<\sinh(4T) = 2\sinh(2T)\cosh(4T).
\]
Thus, for $T<T_1<1.2$ one has $b_2(T) < b_2(1.2) \approx 2.47069<2.48$ and $b_4(T)<b_4(1.2)\approx 1.52666<1.53$. By a direct computation, we obtain
\[
1+0.2^2 + 2\mu_1(0.2)> 2.04 - 0.72\cdot 2.48  - 0.02\cdot 1.53 = 0.2238>0;
\]
\[
1+ 0.2^2 + 2\mu_2(0.2)>1.04 - 0.32\cdot 2.48 - 0.02\cdot 1.53 = 0.2158>0.
\]
\end{proof}
\end{proof}


\end{document}